\documentclass[11pt]{article}
\usepackage{amsthm}
\usepackage{amsfonts}
\usepackage{amsmath}
\usepackage{colonequals}
\usepackage{epsfig}
\usepackage{url}
\newcommand{\GG}{{\cal G}}

\newcommand{\pll}{\text{polylog}}

\newtheorem{theorem}{Theorem}
\newtheorem{corollary}[theorem]{Corollary}
\newtheorem{lemma}[theorem]{Lemma}

\newcommand{\coTW}{\mbox{coTW}}
\newcommand{\tw}{\mbox{tw}}

\title{On classes of graphs with strongly sublinear separators}
\author{Zden\v{e}k Dvo\v{r}\'ak\thanks{Computer Science Institute, Charles University, Prague, Czech Republic. E-mail: {\tt rakdver@iuuk.mff.cuni.cz}.
Supported by the Center of Excellence -- Inst. for Theor. Comp. Sci., Prague (project P202/12/G061 of Czech Science Foundation).}}
\date{}

\begin{document}
\maketitle

\begin{abstract}
For real numbers $c,\varepsilon>0$, let $\GG_{c,\varepsilon}$ denote the class of graphs $G$ such that each subgraph $H$ of $G$
has a balanced separator of order at most $c|V(H)|^{1-\varepsilon}$.  A class $\GG$ of graphs has \emph{strongly sublinear separators}
if $\GG\subseteq \GG_{c,\varepsilon}$ for some $c,\varepsilon>0$.
We investigate properties of such graph classes, leading in particular to an approximate algorithm to determine membership in
$\GG_{c,\varepsilon}$: there exist $c'>0$ such that for each input graph $G$, this algorithm in polynomial
time determines either that $G\in \GG_{c',\varepsilon^2/160}$, or that $G\not\in \GG_{c,\varepsilon}$.
\end{abstract}

A \emph{balanced separator} in a graph $G$ is a set $C\subseteq V(G)$ such that each component of $G-C$ has at most $\tfrac{2}{3}|V(G)|$
vertices (the constant $\tfrac{2}{3}$ is customary but basically arbitrary, any constant smaller than $1$ would give qualitatively the same results).
Balanced separators of small order are of obvious importance in the construction of efficient divide-and-conquer style algorithms~\cite{lt80}.
In these applications, all subgraphs that appear throughout the recursion are required to have small balanced separators.
This motivates us to ask which graphs admit small balanced separators in all their subgraphs.

For real numbers $c>0$ and $0<\varepsilon\le 1$, let $\GG_{c,\varepsilon}$ denote the class of graphs $G$ such that each subgraph $H$ of $G$
has a balanced separator of order at most $c|V(H)|^{1-\varepsilon}$.  A class $\GG$ of graphs has \emph{strongly sublinear separators}
if $\GG\subseteq \GG_{c,\varepsilon}$ for some $c,\varepsilon>0$.  Note that if $c'\le c$ and $\varepsilon'\ge \varepsilon$, then
$\GG_{c',\varepsilon'}\subseteq\GG_{c,\varepsilon}$.  Possibly the best known example of a class with
strongly sublinear separators is the class of planar graphs---Lipton and Tarjan~\cite{lt79} proved that all planar graphs belong
to $\GG_{\sqrt{8},1/2}$ (the constant $1/2$ is tight as shown by planar grids, the constant $\sqrt{8}$ has been subsequently
improved~\cite{betsep}).  More generally, Gilbert et al.~\cite{gilbert} proved that for every surface $\Sigma$ there exists $c>0$ such that
all graphs that can be drawn in $\Sigma$ without crossings belong to $\GG_{c,1/2}$.
All graphs of treewidth at most $t$ belong to $\GG_{t+1,1}$ (and conversely,
Dvo\v{r}\'ak and Norin~\cite{dnorin} proved that all graphs in $\GG_{c,1}$ have treewidth at most $15c$).
Generalizing all these results, Alon et al.~\cite{alon1990separator} and Kawarabayashi and Reed~\cite{kreedsep}
proved that all proper minor-closed classes have strongly sublinear separators.  Furthermore, many geometrically motivated graph classes
(e.g., graphs embeddable in a finite-dimensional Euclidean space with bounded distortion of distances) have this property.

In this paper, we study the question of whether classes of graphs with strongly sublinear separators
admit a common structural description, by leveraging the connection to classes with polynomial expansion (see Section~\ref{sec-prs}
for a definition) discovered by Dvo\v{r}\'ak and Norin~\cite{dvorak2016strongly}.

We start by exploring the consequences of an algorithm to find small balanced separators
by Plotkin, Rao, and Smith~\cite{plotkin} in view of this result, noting a connection
to nowhere-dense classes (Lemma~\ref{lemma-equiv}) and proving that presence of strongly sublinear separators
in all relatively small subgraphs implies their presence in the whole graph (Corollary~\ref{cor-smtobig}).
Section~\ref{sec-prs} contains the relevant definitions and results.

We continue by generalizing the Plotkin-Rao-Smith algorithm, obtaining its weighted variant where
the balanced separators are bounded in terms of prescribed costs of vertices, rather than just
their size (Theorem~\ref{thm-prsgen}).  Through LP duality, we show this implies that every $n$-vertex graph with
strongly sublinear separators contains many small, ``almost disjoint'' subsets of vertices whose
removal makes the treewidth of the rest of the graph polylogarithmic in $n$ (Theorem~\ref{thm-thin}).
This establishes a link to previously studied stronger notion of fractional treewidth-fragility~\cite{twd}.
The weaker property we obtain is sufficient to certify existence of strongly sublinear separators in
all subgraphs whose size is at least polylogarithmic in $n$ (Lemma~\ref{lemma-largesepar}).
These claims are made precise in Section~\ref{sec-frag}.

In Section~\ref{sec-cert}, we show that this property also implies a polynomial-time
algorithm to test presence of given subgraphs of order $O(\log n/\log\log n)$ in graphs with strongly sublinear
separators (Lemma~\ref{lemma-test}).  We use this algorithm together with previously mentioned Corollary~\ref{cor-smtobig} to
obtain a polynomial-time test for presence of strongly sublinear separators in subgraphs
whose size is at most polylogarithmic in $n$, complementing Lemma~\ref{lemma-largesepar}.
Together, this gives an approximate algorithm to verify presence of strongly sublinear separators
in all subgraphs (Theorem~\ref{thm-certif}), as described in the abstract.

Finally, Section~\ref{sec-proofs} is devoted to the proofs of Theorems~\ref{thm-prsgen} and \ref{thm-thin},
postponed in order not to break the flow of the presentation.

\section{Plotkin-Rao-Smith algorithm and its consequences}\label{sec-prs}

A quite general argument for obtaining small balanced separators (which gives another proof that
proper minor-closed classes have strongly sublinear separators) was obtained by Plotkin, Rao, and Smith~\cite{plotkin}.
Let us state their result in a slightly reformulated way, seen to hold by an inspection of their proof
(we do not provide details, as we anyway prove more general Theorem~\ref{thm-prsgen} later).
A \emph{model} of a clique $K_m$ in a graph $G$ is a system $\{V_1, \ldots, V_m\}$
of pairwise-disjoint subsets of $V(G)$ such that $G[V_i]$ is connected for $1\le i\le m$ and $G$ contains an edge with
one end in $V_i$ and the other end in $V_j$ for $1\le i<j\le m$.  The \emph{support} of the model is $V_1\cup \ldots V_m$.
The model has \emph{depth} $d$ if each of the subgraphs
$G[V_1]$, \ldots, $G[V_m]$ has radius at most $d$.
Let $\omega_d(G)$ denote the maximum integer $m$ such that $G$ contains a model of $K_m$ of depth $d$.
For $m_0\ge m$, a model of depth $d$ is \emph{$m_0$-bounded} if
$|V_i|\le m_0d$ for $i=1,\ldots,m$, and in particular the support of the model has size at most $m_0md$.

\begin{theorem}[Plotkin, Rao, and Smith~\cite{plotkin}]\label{thm-prs}
There exists a non-decreasing function $d:\mathbf{N}^2\to\mathbf{N}$ satisfying $d(\ell,n)=O(\ell\log n)$ and a polynomial-time algorithm that,
given an $n$-vertex graph $G$ and integers $\ell,m_0\ge 1$,
returns either an $m_0$-bounded model of $K_{m_0}$ of depth $d\colonequals d(\ell,n)$ in $G$, or
disjoint sets $C,M\subseteq V(G)$ such that
\begin{itemize}
\item $C\cup M$ is a balanced separator in $G$,
\item $|C|\le n/\ell$, and
\item for some $m\le \min(m_0,\omega_d(G))$, $M$ is the support of a $\min(m_0,\omega_d(G))$-bounded model of $K_m$ of depth $d$ in $G$.
\end{itemize}
\end{theorem}
Note that the separator constructed in the previous theorem (with $m_0=n+1$, so that the first conclusion does not
apply) has order at most $n/\ell+d(\ell,n)\omega^2_{d(\ell,n)}(G)=O\bigl(n/\ell+\omega^2_{d(\ell,n)}(G)\ell\log n\bigr)$;
if $G$ is from a proper minor-closed class, then $\omega_d(G)$ is bounded by a constant, and thus we obtain a separator
of order $O(\sqrt{n\log n})$ by setting $\ell=\Theta(\sqrt{n/\log n})$.

The fact that Theorem~\ref{thm-prs} only restricts minors of bounded depth connects strongly sublinear separators
to another concept: \emph{bounded expansion}.  For an integer $r\ge 0$, let $\nabla_r(G)$ denote the
maximum of average degrees of graphs that can be obtained from subgraphs of $G$ by contracting vertex-disjoint
subgraphs of radius at most $r$ (clearly, $\nabla_r(G)\ge\omega_r(G)-1$).  For a function $f$, we say that a class of graphs $\GG$
has \emph{expansion bounded by $f$} if for every integer $r\ge 1$ and every $G\in \GG$, $\nabla_r(G)\le f(r)$;
we say that $\GG$ has \emph{$\omega$-expansion bounded by $f$} if for all such $r$ and $G$, $\omega_r(G)\le f(r)$.
If such a function $f$ exists, then $\GG$ has \emph{bounded expansion} or is \emph{nowhere-dense}, respectively.
We say that $\GG$ has \emph{polynomial expansion} or \emph{polynomial $\omega$-expansion}, respectively,
if this is the case for some polynomial $f$.  

Starting with their introduction by Ne\v{s}et\v{r}il and Ossona de Mendez~\cite{grad1,grad2,npom-nd1},
the notions of bounded expansion and nowhere-density played important roles as models of sparse graph
classes due to their numerous structural and algorithmic properties.
We refer the reader to their book~\cite{nesbook} for a detailed treatment of the theory.
Relevantly to us, Dvo\v{r}\'ak
and Norin~\cite{dvorak2016strongly} proved that a class has strongly sublinear separators if and only if
it has polynomial expansion.  Theorem~\ref{thm-prs} also implies this is equivalent to having
polynomial $\omega$-expansion.
\begin{lemma}\label{lemma-equiv}
For any class $\GG$, the following holds.
\begin{itemize}
\item[\textrm{(a)}] If $\GG$ has $\omega$-expansion bounded by a function $f(r)=O\bigl(r^{1/(2\varepsilon)-2}\bigr)$ for some real number $0<\varepsilon\le 1/4$,
then for some $c>0$, $\GG\subseteq\GG_{c,\varepsilon}$.
\item[\textrm{(b)}] If $\GG\subseteq \GG_{c,\varepsilon}$ for some $c,\varepsilon>0$, then for every $\delta>0$, $\GG$ has expansion bounded by a function
$f(r)=O\bigl(r^{1/\varepsilon+\delta}\bigr)$.
\end{itemize}
In particular, the following claims are equivalent.
\begin{itemize}
\item $\GG$ has polynomial expansion.
\item $\GG$ has polynomial $\omega$-expansion.
\item $\GG$ has strongly sublinear separators.
\end{itemize}
\end{lemma}
\begin{proof}
To prove (a), we give an argument analogous to the one used in Dvo\v{r}\'ak and Norin~\cite[Corollary~2]{dvorak2016strongly}.
Without loss of generality, we can assume that $\GG$ is closed under subgraphs, and thus it suffices to show that graphs in $\GG$ have small balanced separators.
Consider any $n$-vertex graph $G\in\GG$ and apply Theorem~\ref{thm-prs} with $m_0=n+1$
and $\ell=\Theta\bigl(n^\varepsilon\bigr)$, so that $|C|\le n/\ell=O\bigl(n^{1-\varepsilon}\bigr)$. Since $\GG$ has $\omega$-expansion bounded by $O\bigl(r^{1/(2\varepsilon)-2}\bigr)$,
we have $\omega_{d(\ell,n)}(G)=O\bigl(d^{1/(2\varepsilon)-2}(\ell,n)\bigr)$ and
\begin{align*}
|M|&\le d(\ell,n)\omega^2_{d(\ell,n)}(G)=O\bigl(d^{1/\varepsilon-3}(\ell,n)\bigr)\\
&=O\bigl(n^{1-3\varepsilon}\pll(n)\bigr)=O\bigl(n^{1-\varepsilon}\bigr),
\end{align*}
which gives the desired bound on the size of the balanced separator.

The claim (b) was proved by Esperet and Raymond~\cite{espsublin}, strengthening the bound given by Dvo\v{r}\'ak and Norin~\cite{dvorak2016strongly}.
The cycle of equivalences follows by noting that for any function $f$, expansion bounded by $f$ implies $\omega$-expansion bounded by $f+1$.
\end{proof}
The fact that polynomial expansion and polynomial $\omega$-expansion coincide is noteworthy, since for classes whose $\omega$-expansion
grows faster, this is not the case. In particular, there exist nowhere-dense classes that do not have bounded expansion,
and Lemma~\ref{lemma-equiv} shows that their $\omega$-expansion must grow superpolynomially.

Another consequence of Lemma~\ref{lemma-equiv} is that in a class with strongly sublinear separators, the algorithm of Theorem~\ref{thm-prs}
never fails to return a balanced strongly sublinear separator, in the following sense.
\begin{corollary}\label{cor-always}
For all $c>0$ and $0<\varepsilon\le 1$, there exists a polynomial-time algorithm that given an $n$-vertex graph $G\in \GG_{c,\varepsilon}$
returns a balanced separator of order $O(n^{1-\varepsilon/5})$.
\end{corollary}
\begin{proof}
By Lemma~\ref{lemma-equiv}(b), the $\omega$-expansion of $\GG_{c,\varepsilon}$ is bounded by a function $f$ such that $f(r)=O\bigl(r^{1.25/\varepsilon}\bigr)$.
Let $\ell=\Theta\bigl(n^{\varepsilon/5}\bigr)$.  Then the algorithm of Theorem~\ref{thm-prs} (run with $m_0=n+1$) returns a balanced separator of $G$ of order at most
\begin{align*}
O\bigl(n/\ell+f^2(d(\ell,n))d(\ell,n)\bigr)&=O\bigl(n/\ell+(\ell\log n)^{2.5/\varepsilon+1}\bigr)\\
&=O\bigl(n^{1-\varepsilon/5}+n^{1/2+\varepsilon/5}\pll(n)\bigr)\\
&=O\bigl(n^{1-\varepsilon/5}\bigr)
\end{align*}
as required.
\end{proof}
Let us remark that the approximation algorithm of Feige et al.~\cite{feige} could be used instead to give the same result (actually even with
stronger guarantees on the size of the separator).  However, the idea of Corollary~\ref{cor-always} also gives the following surprising fact:
if all small subgraphs of a graph $G$ have strongly sublinear separators, then $G$ itself does as well.
\begin{corollary}\label{cor-smtobig}
For all $c>0$ and $0<\varepsilon,\delta\le 1$, there exists $c'>0$ as follows.
Suppose that $G$ is an $n$-vertex graph.  If all subgraphs of $G$ with at most $n^\delta$ vertices belong to $\GG_{c,\varepsilon}$,
then $G$ belongs to $\GG_{c',\varepsilon\delta/5}$.
\end{corollary}
\begin{proof}
By Lemma~\ref{lemma-equiv}(b), the $\omega$-expansion of $\GG_{c,\varepsilon}$ is bounded by a non-decreasing function $f$ such that $f(r)=O\bigl(r^{1.25/\varepsilon}\bigr)$.
Let $d$ be the function from the statement of Theorem~\ref{thm-prs}.  Choose $\ell=\Theta\bigl(n^{\varepsilon\delta/5}\bigr)$ such that $(f(d(\ell,n))+1)^2d(\ell,n)\le n^\delta$.
Apply the algorithm of Theorem~\ref{thm-prs} with $m_0=f(d(\ell,n))+1$.
Note that if $S$ is a support of an $m_0$-bounded model of depth $d(\ell,n)$ of $K_m$ for some $m\le m_0$,
then $|S|\le n^\delta$ by the choice of $\ell$.  By the assumptions, we have $G[S]\in \GG_{c,\varepsilon}$.
Consequently $m\le \omega_{d(\ell,n)}(G[S])\le f(d(\ell,n))<m_0$,
and thus the first outcome of Theorem~\ref{thm-prs} does not apply.  Hence, we obtain a balanced separator of $G$ of order at most
\begin{align*}
O\bigl(n/\ell+f^2(d(\ell,n))d(\ell,n)\bigr)&=O\bigl(n/\ell+(\ell\log n)^{2.5/\varepsilon+1}\bigr)\\
&\le O\bigl(n^{1-\varepsilon\delta/5}+n^{\delta(1/2+\varepsilon/5)}\pll(n)\bigr)\\
&=O\bigl(n^{1-\varepsilon\delta/5}\bigr)
\end{align*}
as required.
\end{proof}

\section{Weighted separators and fractional treewidth-fragility}\label{sec-frag}

The cornerstone of this paper is the following weighted strengthening of Theorem~\ref{thm-prs}.
For a function $q:V(G)\to \mathbf{Q}_0^+$ (where $\mathbf{Q}_0^+$ are the nonnegative rational numbers) and a set $X\subseteq V(G)$, let us define $q(X)=\sum_{x\in X} q(x)$.
\begin{theorem}\label{thm-prsgen}
There exists a non-decreasing function $d:\mathbf{N}^2\to\mathbf{N}$ satisfying $d(\ell,n)=O(\ell\log n)$ and a polynomial-time algorithm that,
given an $n$-vertex graph $G$, integers $\ell,m_0\ge 1$, and an assignment $q:V(G)\to\mathbf{Q}_0^+$ of nonnegative costs to vertices of $G$,
returns either an $m_0$-bounded model of $K_{m_0}$ of depth $d\colonequals d(\ell,n)$ in $G$, or
disjoint sets $C,M\subseteq V(G)$ such that
\begin{itemize}
\item $C\cup M$ is a balanced separator in $G$,
\item $q(C)\le q(V(G))/\ell$, and
\item for some $m\le \min(m_0,\omega_d(G))$, $M$ is the support of a $\min(m_0,\omega_d(G))$-bounded model of $K_m$ of depth $d$ in $G$.
\end{itemize}
\end{theorem}
Note that the last condition implies $|M|\le (\min(m_0,\omega_d(G)))^2d$.
The proof of this theorem is given in Section~\ref{sec-proofs}.

Before seeing its applications, let us give a few remarks
on Theorem~\ref{thm-prsgen}.
\begin{itemize}
\item There is another way how to introduce weights into Theorem~\ref{thm-prs}: the weights could influence what ``balanced''
means. This was already done by Plotkin, Rao, and Smith~\cite{plotkin}, and Theorem~\ref{thm-prsgen} could be modified in this way
(adding a weight function $w:V(G)\to \mathbf{Q}_0^+$, and requiring that each component $K$ of $G-(C\cup M)$ satisfies $w(V(K))\le \tfrac{2}{3}w(V(G))$
instead of $C\cup M$ being a balanced separator) without any significant changes to its proof.  Since we do not need this device,
we opted not to complicate the statement of Theorem~\ref{thm-prsgen}.

\item The fact that the cost $q(M)$ is not restricted cannot be entirely avoided.  Consider e.g. the case that $G$ is the star
$K_{1,n-1}$ with cost $1/4$ given to the center of the star and cost $3/4$ spread uniformly across
the rays.  For any balanced separator $Z$, either $Z$ contains the center of the star, or at least $n/3$ of the rays,
and thus $q(Z)\ge 1/4$; hence, for $\ell$ larger than $4$, not counting some part of the separator towards its cost is
necessary.

\item For graphs from a class with polynomial $\omega$-expansion, the bound on the size
of $M$ in Theorem~\ref{thm-prsgen} depends polylogarithmically on the number
$n$ of vertices of $G$.  I conjecture that there actually always exists a balanced
separator $C\cup M$ of $G$ with $q(C)\le q(V(G))/\ell$ for some set $M$ of
constant size (dependent on the class and $\ell$, but not on $|V(G)|$).
\end{itemize}

Iterating Theorem~\ref{thm-prsgen}, we can break-up the graph into small pieces combined in a tree-like fashion;
i.e., forming a subgraph of small treewidth.  A \emph{tree decomposition} of a graph $G$ is a pair $(T,\beta)$,
where $T$ is a tree and $\beta$ assigns to each vertex of $T$ a subset of vertices of $G$, such that for every $uv\in E(G)$
there exists $z\in V(T)$ satisfying $\{u,v\}\subseteq \beta(z)$, and for every $v\in V(G)$ the set $\{z\in V(T):v\in\beta(z)\}$
induces a non-empty connected subtree of $T$.  The \emph{width} of the tree decomposition is $\max\{|\beta(z)|:z\in V(T)\}-1$,
and the \emph{treewidth} $\tw(G)$ of $G$ is the minimum width over all tree decompositions of $G$.
The importance of bounded treewidth in the context of classes with strongly sublinear separators
stems from the fact that graphs of treewidth at most $t$ have balanced separators of order at most $t+1$.

\begin{corollary}\label{cor-tw}
There exist non-decreasing functions $b:\mathbf{N}^2\to\mathbf{N}$ and $k:\mathbf{N}^3\to\mathbf{N}$ satisfying $b(\ell,n)=O(\ell\log^2 n)$ and
$k(\ell,\omega,n)=O(\omega^2\ell\log^3 n)$, and
a polynomial-time algorithm that, given an $n$-vertex graph $G$, integers $\ell,m_0\ge 1$,
and an assignment $q:V(G)\to\mathbf{Q}_0^+$ of costs to vertices of $G$, not identically $0$, returns
either an $m_0$-bounded model of $K_{m_0}$ of depth $b(\ell,n)$ in $G$, or
a set $X\subseteq V(G)$ with $q(X)<q(V(G))/\ell$ and a tree decomposition of $G-X$ of width at most $k(\ell, \min(m_0,\omega_{b(\ell,n)}(G)), n)$.
\end{corollary}
\begin{proof}
Let $\ell_0=\ell\lceil \log_{3/2} n\rceil+1$.  Let $d$ be the function from the statement of Theorem~\ref{thm-prsgen} and let us define
$b(\ell,n)=d(\ell_0,n)$.  Let $T$ be a rooted tree and $\theta$, $\gamma$, and $\mu$ functions mapping $V(T)$ to subsets of $V(G)$
obtained as follows.  For the root $r$ of $T$, let $\theta(r)=V(G)$.  Considering a vertex $v\in V(T)$, let us apply the algorithm
of Theorem~\ref{thm-prsgen} to $G[\theta(v)]$, with $\ell=\ell_0$ and $q$ restricted to $\theta(v)$.  If the outcome is
an $m_0$-bounded model of $K_{m_0}$ of depth $d(\ell_0,n)=b(\ell,n)$ in $G[\theta(v)]$,
we return this model.  Otherwise, let $C$ and $M$ be as in the statement of Theorem~\ref{thm-prsgen}.  We set $\gamma(v)=C$, $\mu(v)=M$,
and for each component $K$ of $G[\theta(v)]-(C\cup M)$, we create a child $u$ of $v$ with $\theta(u)=V(K)$.  We repeat this construction
to obtain $T$, $\theta$, $\gamma$, and $\mu$.

Let $X=\bigcup_{v\in V(T)} \gamma(v)$.  For $v\in V(T)$, let $\beta(v)$ be the union of the sets $\mu(u)$ over all ancestors $u$ of $v$ in $T$
(including $v$ itself).  Then $(T,\beta)$ is a tree decomposition of $G-X$.  Note that $T$ has depth at most $\lceil\log_{3/2} n\rceil$.
Since $|\mu(u)|\le (\min(m_0,\omega_{b(\ell,n)}(G)))^2b(\ell,n)$ for all $u\in V(T)$, we conclude that the width of the decomposition is
$O((\min(m_0,\omega_{b(\ell,n)}(G)))^2b(\ell,n)\log n)$, and we define the function $k$ according to this bound.

For $0\le i\le \lceil\log_{3/2} n\rceil-1$, let $V_i$ denote the set of vertices
of $T$ at distance $i$ from the root.  Note that for distinct $u,v\in V_i$, the sets $\theta(u)$ and $\theta(v)$ are disjoint.
We conclude that $q(\bigcup_{v\in V_i} \gamma(v))\le q(V(G))/\ell_0$, and since $V(T)=\bigcup_{i=0}^{\lceil\log_{3/2} n\rceil-1} V_i$,
we have $q(X)\le \lceil \log_{3/2} n\rceil q(V(G))/\ell_0<q(V(G))/\ell$, as required.
\end{proof}

We will actually need a dual form of Corollary~\ref{cor-tw}, related to the concept of \emph{fractional treewidth-fragility}
introduced in~\cite{twd}.
Let $G$ be a graph and let $\mathcal{X}$ be a collection of subsets of $V(G)$.
For $\varepsilon>0$, we say that a probability distribution on $\mathcal{X}$ is \emph{$\varepsilon$-thin} if
a set $X\in \mathcal{X}$ chosen at random from this distribution satisfies $\text{Pr}[v\in X]\le \varepsilon$ for all $v\in V(G)$.
The \emph{support} of this distribution consists of the elements of $\mathcal{X}$ with non-zero probability.
For $k\ge 0$, let $\coTW_k(G)$ denote the collection of all sets $X\subseteq V(G)$ such that
$\tw(G-X)\le k$.  We say that a class $\GG$ is \emph{fractionally $(\varepsilon,k)$-treewidth-fragile} if for each graph $G\in \GG$
there exists an $\varepsilon$-thin probability distribution on $\coTW_k(G)$.  A class $\GG$ is \emph{fractionally treewidth-fragile}
if for all $\varepsilon>0$ there exists $k\ge 0$ such that $\GG$ is fractionally $(\varepsilon,k)$-treewidth-fragile.

As an example, consider a connected planar graph $G$.  Let $v$ be any vertex of $G$, and for $i\ge 0$, let $V_i$ be the set
of vertices of $G$ at distance exactly $i$ from $v$.  For $0\le j<t$, let $V_{j,t}=\bigcup_{i\bmod t=j} V_i$.
A result of Robertson and Seymour~\cite{rs3} implies that the graph $G-V_{j,t}$ has treewidth at most $3t+1$,
and thus $V_{j,t}\in\coTW_{3t+1}(G)$.  Let us assign each of the sets $V_{j,t}$ for $0\le j\le t-1$ probability $1/t$,
and all other sets in $\coTW_{3t+1}(G)$ probability $0$.  Since the sets $V_{0,t}$, \ldots, $V_{t-1,t}$ are pairwise
disjoint, if a set $X$ is chosen at random from this distribution, we have $\text{Pr}[v\in X]\le 1/t$ for all $v\in V(G)$.
Hence, the class of planar graphs is fractionally $(1/t,3t+1)$-treewidth-fragile.  Consequently, for every $\varepsilon>0$,
the class of planar graphs is fractionally $(\varepsilon, 3\lceil 1/\varepsilon\rceil+1)$-treewidth-fragile, and thus the class of planar
graphs is fractionally treewidth-fragile.

Let us remark that the word ``fractional'' in the definition of fractional treewidth-fragility
refers to the fact that the sets in the support of the distribution do not need to be pairwise disjoint; hence, planar graphs
are actually examples of classes that are ``treewidth-fragile'', in a non-fractional sense.  Three-dimensional grids are among examples
of natural graph classes that are fractionally treewidth-fragile, but not treewidth-fragile in the non-fractional sense; see~\cite{twd,gridtw}
for more details.

Fractional treewidth-fragility has applications in algorithmic design, especially regarding approximation algorithms (essentially,
a set $X$ sampled from an $\varepsilon$-thin probability distribution on $\coTW_k(G)$ will likely intersect an optimal solution
in only a small fraction of vertices, and one is often able to efficiently recover the large part of the solution contained in $G-X$ using the
fact that this graph has small treewidth); see~\cite{twd} for more details.  Also, fractional treewidth-fragility implies sublinear separators (see~\cite[Lemma 14]{twd}),
and in~\cite{twd}, I conjectured a converse: strongly sublinear separators imply fractional treewidth-fragility.  If in Corollary~\ref{cor-tw},
the treewidth of $G-X$ did not depend on the number of vertices of $G$, this would imply this conjecture.  As it stands,
Corollary~\ref{cor-tw} only implies a weakening of the claim we now describe.

\begin{theorem}\label{thm-thin}
There exist non-decreasing functions $b:\mathbf{N}^2\to\mathbf{N}$ and $k_0:\mathbf{N}^3\to\mathbf{N}$ satisfying $b(\ell,n)=O(\ell\log^2 n)$ and
$k_0(\ell,\omega,n)=O(\omega^2\ell\log^3 n)$, and
a polynomial-time algorithm that, given an $n$-vertex graph $G$ and integers $\ell,m_0\ge 1$,
returns either an $m_0$-bounded model of $K_{m_0}$ of depth $b(\ell,n)$ in $G$, or
a $(1/\ell)$-thin probability distribution on $\coTW_t(G)$ with support of size at most $n$,
where $t=k_0(\ell, \min(m_0,\omega_{b(\ell,n)}(G)), n)$.
\end{theorem}
The proof of this theorem is given in Section~\ref{sec-proofs}.

In particular, for classes of graphs with strongly sublinear separators, Lemma~\ref{lemma-equiv} and Theorem~\ref{thm-thin} give the following.
Let $k:\mathbf{N}^2\to\mathbf{N}$ be a non-decreasing function.  We say that an $n$-vertex graph $G$ is \emph{fractionally $k$-treewidth-fragile}
if for $\ell=1,\ldots,n$, there exists a $(1/\ell)$-thin probability distribution on $\coTW_{k(\ell,n)}(G)$.
\begin{corollary}\label{cor-thin}
Let $c>0$ and $0<\varepsilon\le 1$ be real numbers.  There exists a non-decreasing function $k:\mathbf{N}^2\to\mathbf{N}$
satisfying $k(\ell, n)=O\bigl((\ell\log^2 n)^{4/\varepsilon}\bigr)$ such that all graphs in
$\GG_{c,\varepsilon}$ are fractionally $k$-treewidth-fragile.  Furthermore, there exists a polynomial-time
algorithm that for each $n$-vertex graph $G$ and an integer $1\le \ell\le n$ either returns a $(1/\ell)$-thin probability distribution on
$\coTW_{k(\ell,n)}(G)$ with support of size at most $n$, or shows that $G\not\in \GG_{c,\varepsilon}$.
\end{corollary}
\begin{proof}
By Lemma~\ref{lemma-equiv}(b), the $\omega$-expansion of $\GG_{c,\varepsilon}$ is bounded by a non-decreasing function $f$ such
that $f(r)=O\bigl(r^{1.25/\varepsilon}\bigr)$.  For a given $n$-vertex graph $G$
and an integer $\ell\in\{1,\ldots,n\}$, let us apply the algorithm of Theorem~\ref{thm-thin} with $m_0=f(b(\ell,n))+1$.
If the algorithm returns a model of $K_{m_0}$ of depth $b(\ell,n)$ in $G$, then
$\omega_{b(\ell,n)}(G)\ge m_0>f(b(\ell,n))$, and thus $G\not\in \GG_{c,\varepsilon}$.
Otherwise, the algorithm returns a $(1/\ell)$-thin probability distribution on $\coTW_t(G)$ with support of size at most $n$,
where
\begin{align*}
t&=k_0(\ell, \min(m_0,\omega_{b(\ell,n)}(G)), n)\le k_0(\ell,f(b(\ell,n)),n)\\
&=k_0\bigl(\ell,O\bigl((\ell\log^2 n)^{1.25/\varepsilon}\bigr),n\bigr)\\
&=O\bigl((\ell\log^2 n)^{2.5/\varepsilon}\ell\log^3 n\bigr)=O\bigl((\ell\log^2n)^{4/\varepsilon}\bigr).
\end{align*}
We can choose the function $k$ accordingly.
\end{proof}

Fractional $k$-treewidth-fragility for the function $k$ from Corollary~\ref{cor-thin} does not necessarily imply strongly sublinear separators.  However, it does imply existence of
small separators in subgraphs of at least polylogarithmic size.

\begin{lemma}\label{lemma-largesepar}
For a real number $t\ge 0$ and a non-decreasing function $k:\mathbf{N}^2\to\mathbf{N}$ satisfying $k(\ell, n)=O\bigl((\ell\log^2 n)^t\bigr)$,
there exists a real number $c>0$ as follows.
Suppose that an $n$-vertex graph $G$ is fractionally $k$-treewidth-fragile.
If $H$ is a subgraph of $G$ with $m\ge \log^{4t}n$ vertices, then $H$ has a balanced separator of order at most $cm^{1-1/(2t+2)}$.
\end{lemma}
\begin{proof}
Let $\ell=\lceil m^{1/(2t+2)}\rceil$. For $X\in \coTW_{k(\ell,n)}(G)$ chosen at random from a $(1/\ell)$-thin probability distribution,
the expected value of $|X\cap V(H)|$ is at most $m/\ell$.  Hence, there exists $X\subseteq V(G)$ such that
$|X\cap V(H)|\le m/\ell$ and $G-X$ has treewidth at most $k(\ell,n)$.  Consequently, $H-X$ also has
treewidth at most $k(\ell,n)$, and thus $H-X$ has a balanced separator $C_0$ of order at most
$k(\ell,n)+1=O\bigl(m^{1/2-1/(2t+2)}\log^{2t} n\bigr)=O\bigl(m^{1-1/(2t+2)}\bigr)$, using the assumption that $\log^{2t} n\le m^{1/2}$.
We conclude that $(X\cap V(H))\cup C_0$ is a balanced separator in $H$ of order at most $m/\ell+k(\ell,n)+1=O\bigl(m^{1-1/(2t+2)}\bigr)$.
\end{proof}

\section{Certification of strongly sublinear separators}\label{sec-cert}

Many of the polynomial-time algorithms for fractionally treewidth-fragile classes from~\cite{twd} become only pseudopolynomial (with time complexity
$n^{\pll(n)}$) or worse when used for fractionally $k$-treewidth-fragile graphs with $k(\ell, n)=O\bigl((\ell\log^2 n)^t\bigr)$.
An exception is the following subgraph testing algorithm.

\begin{lemma}\label{lemma-test}
Let $t\ge 0$ be a real number and let $k:\mathbf{N}^2\to\mathbf{N}$ be a non-decreasing function satisfying $k(\ell, n)=O\bigl((\ell\log^2 n)^t\bigr)$.
Let $G$ be an $n$-vertex fractionally $k$-treewidth-fragile graph, and suppose a $(1/\ell)$-thin probability distribution on
$\coTW_{k(\ell,n)}(G)$ with support of size at most $n$ is given for all $\ell\in\{1,\ldots,n\}$.
Then it is possible to decide whether an $m$-vertex graph $H$ is a subgraph of $G$ in time $O\bigl(e^{10m\log m}n^3\bigr)$.
\end{lemma}
\begin{proof}
Clearly, we can assume that $m<n$, as if $V(H)=V(G)$, then the trivial algorithm testing all $m!$ bijections from $V(H)$ to $V(G)$
suffices, and if $m>n$, then the answer is always no.

Let $\ell=m+1$.  Let $\{X_1, \ldots,X_{n'}\}\subseteq \coTW_{k(\ell,n)}(G)$ for some $n'\le n$ be the support of the $(1/\ell)$-thin probability distribution on $\coTW_{k(\ell,n)}(G)$
which we are given.  If $H$ is a subgraph of $G$, then for $X$ chosen at random from this distribution
the expected value of $|X\cap V(H)|$ is at most $m/\ell<1$, and thus the probability that $V(H)\cap X=\emptyset$ is non-zero.
Hence, $H$ is a subgraph of $G$ if and only if $H$ is a subgraph of one of the graphs $G-X_1$, \ldots, $G-X_{n'}$.
Since these graphs have treewidth at most $k(\ell,n)=O\bigl((m\log^2 n)^t\bigr)$, we can determine whether $H$ is a subgraph in each of them
in time $O\bigl(k(\ell,n)^{2m}n\bigr)=O\bigl(m^{2tm}\log^{4tm}n\cdot n\bigr)=O\bigl(m^{10tm}n^2\bigr)$ using a standard dynamic programming algorithm.
\end{proof}

Finally, we turn attention to the question of testing whether a graph belongs to a class with strongly sublinear separators.
Testing exact membership in a class $\GG_{c,\varepsilon}$ for some given $c,\varepsilon>0$ is likely hard (determining the
smallest size of a balanced separator is NP-hard~\cite{feige2006finding}, but this is a slightly different problem).
It is possible to approximate the smallest size of a balanced separator~\cite{feige},
however it is not clear whether this is helpful, as to test the (approximate) membership of a graph $G$ in $\GG_{c,\varepsilon}$, one needs to verify that
all (exponentially many) subgraphs of $G$ have small separators.  Hence, the following result is of interest.

\begin{theorem}\label{thm-certif}
For every $c,\varepsilon>0$, there exist $c'>0$ and a polynomial-time algorithm that for each input graph $G$,
determines either that $G\in \GG_{c',\varepsilon^2/160}$, or that $G\not\in \GG_{c,\varepsilon}$.
\end{theorem}
\begin{proof}
Let $n=|V(G)|$.  First, we run the algorithm of Corollary~\ref{cor-thin} for $\ell=1,\ldots,n$.  This either shows that $G\not\in \GG_{c,\varepsilon}$,
or gives us $(1/\ell)$-thin probability distributions on $\coTW_{k(\ell,n)}(G)$ with supports of size at most $n$ for $\ell\in\{1,\ldots,n\}$,
where $k(\ell,n)=O\bigl((\ell\log^2 n)^{4/\varepsilon}\bigr)$.
By Lemma~\ref{lemma-largesepar}, this shows that each subgraph of $G$ with $m\ge \log^{16/\varepsilon} n$ vertices has a balanced separator of order
$O\bigl(m^{1-\varepsilon/10}\bigr)$.

Next, we test whether all subgraphs of $G$ with $m\le \log^{1/2} n$ vertices belong to $\GG_{c,\varepsilon}$.  Note that there are only $O(n)$ non-isomorphic graphs $H$
with at most $\log^{1/2} n$ vertices; for each such graph $H$, we can test whether it belongs to $\GG_{c,\varepsilon}$ by brute-force testing
all its induced subgraphs and subsets of their vertices in time $O\bigl(m^23^m\bigr)=O(n)$; and we can test whether $H\subseteq G$
in time $O(n^4)$ according to Lemma~\ref{lemma-test}.  Hence, this part can be carried out in total time $O(n^5)$.

Now, if $H$ is a subgraph of $G$ with $m<\log^{16/\varepsilon} n$ vertices, then according to the previous paragraph,
all its subgraphs with at most $m^{\varepsilon/32}$ vertices belong to $\GG_{c,\varepsilon}$, and by Corollary~\ref{cor-smtobig},
$H$ has a balanced separator of order $O\bigl(m^{1-\varepsilon^2/160}\bigr)$.

Consequently, each $m$-vertex subgraph of $G$ has a balanced separator of order $O\bigl(m^{1-\varepsilon/10}\bigr)$ or $O\bigl(m^{1-\varepsilon^2/160}\bigr)$
depending on whether $m\ge \log^{16/\varepsilon} n$ or not, and thus $G\in \GG_{c',\varepsilon^2/160}$ for some $c'>0$.
\end{proof}

\section{Proofs of Theorems~\ref{thm-prsgen} and \ref{thm-thin}}\label{sec-proofs}

The proof of Theorem~\ref{thm-prsgen} essentially follows the argument of Plotkin, Rao, and Smith~\cite{plotkin}, with
a few minor modifications.  Let us start with the key lemma, showing that either a graph contains a small cost separator,
or it has bounded radius.
\begin{lemma}\label{lemma-seporexp}
There exists a polynomial-time algorithm that, given a graph $G$ with at most $n$ vertices, integers $\ell_0,r\ge 1$,
and an an assignment $q:V(G)\to\mathbf{Q}^+$ of positive costs to vertices of $G$ such that $q(v)\ge q(V(G))/r$ for every $v\in V(G)$,
returns either a vertex $v_0\in V(G)$ such that each other vertex is at distance at most $\lfloor 2+\ell_0\log_2 rn\rfloor$ from $v_0$, or
a partition of $V(G)$ to parts $C_2$, $D$ and $E$ such that there are no edges between $D$ and $E$,
$D\neq\emptyset\neq E$, $q(C_2)\le q(D)/\ell_0$ and $q(C_2)\le q(E)/\ell_0$.
\end{lemma}
\begin{proof}
If $G$ is not connected, then we can let $C_2=\emptyset$, let $D$ be the vertex set of a component of $G$ and let $E=V(G)\setminus D$.
If $|V(G)|=1$, then we can return the only vertex of $G$ as $v_0$.
Hence, assume that $G$ is connected and has at least two vertices.

Let $v_0$ be a vertex of $G$ of maximum cost.  For any integer $i$, let $V_i$, $S_i$, and $L_i$ be the set of vertices of $G$
at distance exactly $i$, less than $i$, and more than $i$ from $v_0$.  Let $d\ge 1$ be the maximum index such that $V_d\neq\emptyset$.
If there exists $1\le i\le d-1$ such that $q(V_i)\le q(S_i)/\ell_0$ and $q(V_i)\le q(L_i)/\ell_0$, then we can return
$C_2=V_i$, $D=S_i$ and $E=L_i$.
Hence, we can assume that for $1\le i\le d-1$, we have either $q(V_i)>q(S_i)/\ell_0$ or $q(V_i)>q(L_i)/\ell_0$.
In the former case, $q(S_{i+1})>(1+1/\ell_0)q(S_i)$ and $q(L_i)<q(L_{i-1})$.
In the latter case, $q(S_{i+1})>q(S_i)$ and $q(L_i)<(1-1/\ell_0)q(L_{i-1})$.

Since $q(S_1)=q(v_0)\ge q(V(G))/n$ and $q(S_d)<q(V(G))$, the number $a$ of indices $i$ such that
$1\le i\le d-1$ and $q(S_{i+1})>(1+1/\ell_0)q(S_i)$ satisfies $(1+1/\ell_0)^a<n$.  Since $1+1/\ell_0\ge 2^{1/\ell_0}$,
we conclude $a<\ell_0 \log_2 n$.  Since $q(L_0)<q(V(G))$ and $q(L_{d-1})\ge q(V(G))/r$, the number $b$ of indices
$i$ such that $1\le i\le d-1$ and $q(L_i)<(1-1/\ell_0)q(L_{i-1})$ satisfies $(1-1/\ell_0)^b>1/r$.
Since $(1-1/\ell_0)\le 2^{-1/\ell_0}$, it follows that $b<\ell_0\log_2 r$.  Consequently, $d\le a+b+2\le 2+\ell_0\log_2 rn$,
and each vertex is at distance at most $d$ from $v_0$; hence, the algorithm can return $v_0$.
\end{proof}

We are now ready to prove the theorem.

\begin{proof}[Proof of Theorem~\ref{thm-prsgen}]
Let us define $d(\ell,n)=\lfloor 2+2\ell\log_2 (2\min(\ell,n) n^2)\rfloor=O(\ell\log n)$.
Clearly, we can assume $\ell\le n$, since otherwise $d(\ell,n)\ge n$ and we can return $C=\emptyset$ and $M$ equal to the vertex
set of the largest component of $G$ (forming the support of a $1$-bounded model of $K_1$ of depth $n$).

Let $C_0$ be the set of vertices $v\in V(G)$ such that $q(v)\le q(V(G))/(2\ell n)$.
Clearly, $q(C_0)\le q(V(G))/(2\ell)$.  Let $G'=G-C_0$.
The algorithm maintains a partition of the vertices of $G'$ to sets $A$, $M$, and $R=V(G')\setminus (A\cup M)$
satisfying the following invariants.
\begin{itemize}
\item[(i)] $|A|\le \tfrac{2}{3}|V(G)|$,
\item[(ii)] if $C_1$ is the set of vertices of $R$ with a neighbor in $A$, then $q(C_1)\le q(A)/(2\ell)$, and
\item[(iii)] for some integer $m\le \min(m_0,\omega_d(G))$, $M$ is the support of a $\min(m_0,\omega_d(G))$-bounded model of $K_m$ of depth $d$ in $G$.
\end{itemize}

Initially, we have $A=M=\emptyset$.  We repeat the following steps.
\begin{itemize}
\item If $m=m_0$, we can return $M$ and stop.  Hence, suppose that $m\le m_0-1$.

\item If $|R\setminus C_1|\le \tfrac{2}{3}|V(G)|$, then we can return
the subsets $C=C_0\cup C_1$ and $M$ of $V(G)$, which clearly satisfy the requirements
of the theorem.  Hence, suppose that $|R\setminus C_1|>\tfrac{2}{3}|V(G)|$, and thus $|A|<\tfrac{1}{3}|V(G)|$
and $|M|<\tfrac{1}{3}|V(G)|$.

\item Let $M_1$, \ldots, $M_m$ be the parts of the partition forming the model of $K_m$ with support $M$.
If there exists $i$ such that $1\le i\le m$ and $M_i$ has no neighbors in $R$, then let $A\colonequals A\cup M_i$
and $M\colonequals M\setminus M_i$.  Since $|A|+|M|<\tfrac{2}{3}|V(G)|$, this clearly preserves the invariants.

\item Hence, we can assume that for $1\le i\le m$, the set $N_i$ of neighbors of $M_i$ in $R$ is non-empty.
Let us apply Lemma~\ref{lemma-seporexp} to $G'[R]$ with $\ell_0=2\ell$ and $r=2\ell n$.
If the result is a vertex $v_0$ at distance at most $\lfloor 2+\ell_0\log_2 rn\rfloor=d$ from all other vertices
of $G'[R]$, then let $M_{m+1}$ be the union of the vertex sets of the shortest paths from $v_0$ to $N_1$, \ldots, $N_m$
and $M\colonequals M\cup M_{m+1}$.  Clearly, this preserves the invariants.

\item Finally, suppose that the result of invocation of the algorithm from Lemma~\ref{lemma-seporexp} is a partition of $R'$ into sets $C_2$, $D$, and $E$
such that $D\neq\emptyset\neq E$, $q(C_2)\le q(D)/(2\ell)$ and $q(C_2)\le q(E)/(2\ell)$, and there are no edges between $D$ and $E$.
By symmetry, we can assume that $|D|\le |E|$.  We let $A\colonequals A\cup D$.
Note that the set of neighbors of this new set $A$ in $R\setminus D$ is a subset of $C_1\cup C_2$,
and thus its cost is at most $q(A)/(2\ell)$.  Furthermore, denoting by $A'$ the old set $A$, we have
$|A|=|A'|+|D|\le |A'|+|R|/2\le |A'|+(|V(G)|-|A'|)/2=(|V(G)|+|A'|)/2<\tfrac{2}{3}|V(G)|$,
since $|A'|<\tfrac{1}{3}|V(G)|$.  Hence, the invariants are again preserved.
\end{itemize}

Since in each step we either increase $|A|$ or decrease $|R|$, and $|A|$ never decreases and $|R|$ never increases,
this algorithm terminates in at most $2n$ iterations.
\end{proof}

Theorem~\ref{thm-thin} follows from Corollary~\ref{cor-tw} by linear programming duality and ellipsoid method.

\begin{proof}[Proof of Theorem~\ref{thm-thin}]
Let $b$ and $k$ be as in Corollary~\ref{cor-tw}.  Let us form a linear program as follows.
For each $X\in \coTW_k(G)$, let us introduce a variable $p(X)\ge 0$.  Additionally, let $\varepsilon$ be a variable.
For each vertex $v\in V(G)$, we have the following
constraint:
$$\sum_{X\in \coTW_k(G), v\in X} p(X)\le \varepsilon,$$
and additionally
$$\sum_{X\in \coTW_k(G)} p(X)=1.$$
An $(1/\ell)$-thin probability distribution on $\coTW_k(G)$ exists if and only if the minimum $\varepsilon$ subject to these constraints
is at most $1/\ell$.

Consider the dual to this linear program, with variables $q(v)$ for $v\in V(G)$ and another variable $s$:
\begin{align*}
q(v)&\ge 0&&\text{for $v\in V(G)$}\\
\sum_{v\in V(G)} q(v)&=1\\
\sum_{v\in X} q(v)&\ge s&&\text{for $X\in\coTW_k(G)$}\\
\text{maximize}&\text{ $s$}
\end{align*}
Given any assignment of values $q:V(G)\to\mathbf{Q}_0^+$ such that $\sum_{v\in V(G)} q(v)=1$, the algorithm of Corollary~\ref{cor-tw}
either returns an $m_0$-bounded model of $K_{m_0}$ of depth $b(\ell,n)$ in $G$ (in which case we can return this model
and stop), or a set $X\in \coTW_k(G)$ such that $\sum_{v\in X} q(v)<1/\ell$.  Run the ellipsoid method algorithm
for the polytope defined by inequalities
\begin{align*}
q(v)&\ge 0&&\text{for $v\in V(G)$}\\
\sum_{v\in V(G)} q(v)&=1\\
\sum_{v\in X} q(v)&\ge 1/\ell&&\text{for $X\in\coTW_k(G)$,}
\end{align*}
using the algorithm of Corollary~\ref{cor-tw} as a separation oracle.  Unless at some point we stop due to the discovery of
an $m_0$-bounded model of $K_{m_0}$, the conclusion will necessarily be that this polytope is empty.  During the run of the ellipsoid
method algorithm, the algorithm of Corollary~\ref{cor-tw} will return polynomially many sets $X_1, \ldots, X_t\in\coTW_k(G)$.
Let $\mathcal{X}=\{X_1,\ldots,X_t\}$.  Clearly, the polytope defined by
\begin{align*}
q(v)&\ge 0&&\text{for $v\in V(G)$}\\
\sum_{v\in V(G)} q(v)&=1\\
\sum_{v\in X} q(v)&\ge 1/\ell&&\text{for $X\in \mathcal{X}$}
\end{align*}
is also empty, and thus the optimum of the linear program
\begin{align*}
q(v)&\ge 0&&\text{for $v\in V(G)$}\\
\sum_{v\in V(G)} q(v)&=1\\
\sum_{v\in X} q(v)&\ge s&&\text{for $X\in\mathcal{X}$}\\
\text{maximize}&\text{ $s$}
\end{align*}
is smaller than $1/\ell$.  Since this linear program has polynomial size, we can in polynomial
time find its optimal solution.  This optimum is achieved in a vertex, and since the dimension of the program (accounting
for the equality constraint $\sum_{v\in V(G)} q(v)=1$) is $n$, this vertex is defined by $n$ of the constraints of form $q(v)\ge 0$ or
$\sum_{v\in X} q(v)\ge s$ being tight.  Let $\mathcal{X}'\subseteq \mathcal{X}$ consist of (at most $n$) sets $X$ whose tight constraints
are involved in the definition of the optimum vertex.  Hence, the linear program
\begin{align*}
q(v)&\ge 0&&\text{for $v\in V(G)$}\\
\sum_{v\in V(G)} q(v)&=1\\
\sum_{v\in X} q(v)&\ge s&&\text{for $X\in\mathcal{X}'$}\\
\text{maximize}&\text{ $s$}
\end{align*}
has the same optimum.  By duality, the optimum of
\begin{align*}
p(X)&\ge 0&&\text{for $X\in\mathcal{X}'$}\\
\sum_{X\in \mathcal{X}'} p(X)&=1\\
\sum_{X\in \mathcal{X}', v\in X} p(X)&\le \varepsilon&&\text{for $v\in V(G)$}\\
\text{minimize}&\text{ $\varepsilon$}
\end{align*}
is smaller than $1/\ell$.  Again, we can find an optimal solution to this program in polynomial time.
Setting $p(X)$ according to this solution and $p(X)=0$ for $X\in \coTW_k(G)\setminus\mathcal{X}'$,
we obtain a $(1/\ell)$-thin probability distribution on $\coTW_k(G)$ with support of size at most $n$.
\end{proof}

\bibliographystyle{siam}
\bibliography{almfrag}

\end{document}